\pdfoutput=1
\RequirePackage{ifpdf}
\ifpdf 
\documentclass[pdftex]{sigma}
\else
\documentclass{sigma}
\fi

\numberwithin{equation}{section}

\newtheorem{Theorem}{Theorem}[section]
\newtheorem{Lemma}[Theorem]{Lemma}
\newtheorem{Proposition}[Theorem]{Proposition}
 { \theoremstyle{definition}
\newtheorem{Example}[Theorem]{Example}
\newtheorem{Remark}[Theorem]{Remark} }

\begin{document}

\allowdisplaybreaks

\newcommand{\arXivNumber}{1609.06247}

\renewcommand{\PaperNumber}{049}

\FirstPageHeading

\ShortArticleName{On the Spectra of Real and Complex Lam\'e Operators}

\ArticleName{On the Spectra of Real and Complex Lam\'e Operators}

\Author{William A.~HAESE-HILL~$^\dag$, Martin A.~HALLN\"AS~$^\ddag$ and Alexander P.~VESELOV~$^\dag$}

\AuthorNameForHeading{W.A.~Haese-Hill, M.A.~Halln\"as and A.P.~Veselov}

\Address{$^\dag$~Department of Mathematical Sciences, Loughborough University, \\
\hphantom{$^\dag$}~Loughborough LE11 3TU, UK}
\EmailD{\href{mailto:wahhill@gmail.com}{wahhill@gmail.com}, \href{mailto:A.P.Veselov@lboro.ac.uk}{A.P.Veselov@lboro.ac.uk}}

\Address{$^\ddag$~Department of Mathematical Sciences, Chalmers University of Technology and\\
\hphantom{$^\ddag$}~the University of Gothenburg, SE-412 96 Gothenburg, Sweden}
\EmailD{\href{mailto:hallnas@chalmers.se}{hallnas@chalmers.se}}

\ArticleDates{Received April 04, 2017, in f\/inal form June 21, 2017; Published online July 01, 2017}

\Abstract{We study Lam\'e operators of the form
\begin{gather*}
L = -\frac{d^2}{dx^2} + m(m+1)\omega^2\wp(\omega x+z_0),
\end{gather*}
with $m\in\mathbb{N}$ and $\omega$ a half-period of $\wp(z)$. For rectangular period lattices, we can choose $\omega$ and $z_0$ such that the potential is real, periodic and regular. It is known after Ince that the spectrum of the corresponding Lam\'e operator has a band structure with not more than $m$ gaps. In the f\/irst part of the paper, we prove that the opened gaps are precisely the f\/irst $m$ ones. In the second part, we study the Lam\'e spectrum for a generic period lattice when the potential is complex-valued. We concentrate on the $m=1$ case, when the spectrum consists of two regular analytic arcs, one of which extends to inf\/inity, and brief\/ly discuss the $m=2$ case, paying particular attention to the rhombic lattices.}


\Classification{34L40; 47A10; 33E10}

\section{Introduction}
The Lam\'e equation
\begin{gather*}
-\frac{{\rm d}^2\psi}{{\rm d}z^2} + m(m+1) \wp(z)\psi = \lambda\psi,
\end{gather*}
where $m\in\mathbb N$ and $\wp(x)$ is Weierstrass' elliptic function, satisfying
\begin{gather*}
(\wp')^2 = 4(\wp-e_1)(\wp-e_2)(\wp-e_3),
\end{gather*}
is a classical object of 19\textsuperscript{th} century mathematics.

It surprisingly appeared again at the early age of quantum mechanics in the work by Kramers and Ittmann \cite{Kramers1,Kramers2}, who discovered that the Lam\'e equation is closely related to the quantum top by showing that the corresponding Schr\"{o}dinger equation is separable in the elliptic coordinate system and that the resulting dif\/ferential equations are of Lam\'{e} form (see, e.g.,~\cite{GV}). It is interesting that initially the equation was derived by Gabriel Lam\'e in 1837 as a result of separation of variables for the Laplace equation in elliptic coordinates, see~\cite{Lame}, so one might argue that its origin was quantum from the very beginning.

Viewed as an equation in the complex domain, its solutions, which have a number of remarkable properties, were described explicitly by Hermite and Halphen, see, e.g.,~\cite{WW}. Note that for~$x$ on the real line the potential $\wp(x)$ has singularities. If, however, $e_1$, $e_2$, $e_3$ are all real, one can make a pure imaginary half-period shift by $z_0 = \omega_3$ and consider the Lam\'e operator
\begin{gather}\label{intro_jac_lame}
L = -\frac{{\rm d}^2}{{\rm d}x^2} + m(m+1)\wp(x+z_0),
\end{gather}
whose potential is real, $2\omega_1$-periodic and regular on the whole of~$\mathbb{R}$. Consequently, one can apply Bloch--Floquet theory, which implies that the spectrum has a band structure, as described, e.g., in~\cite{ReedSimon}.

Generically, the spectrum of a periodic Schr\"odinger operator (or Hill operator) on the real line consists of a countably inf\/inite number of bands. It was Ince~\cite{Ince} who f\/irst pointed out the remarkable fact that the spectrum of $L$ has a band structure with not more than $m$ gaps. Erd\'elyi~\cite{erdelyi} later showed that in fact all $m$ gaps are open.

Nowadays, this is just the simplest example in the large class of f\/inite-gap operators discovered in the 1970s, see~\cite{DMN,L, N}. It turned out that all such operators can be described explicitly in terms of hyperelliptic Riemann theta functions, with the Lam\'e operator~\eqref{intro_jac_lame} corresponding to the elliptic case. However, it seems that the question of exactly which gaps in the spectrum are open has not been explicitly discussed in the literature even in this case.

The f\/irst result of this paper, obtained in Section~\ref{Sec:realLame}, demonstrates that in the Lam\'e case it is precisely the \emph{first} $m$ gaps that are open. Although this fact may not be surprising for experts, we could not f\/ind a rigorous proof in the literature. For $m=1$, the result is illustrated in Fig.~\ref{Fig:oneGap}, showing that all the closed gaps are indeed located in the inf\/inite spectral band.

\begin{figure}[t]\centering
\includegraphics[width=0.75\textwidth]{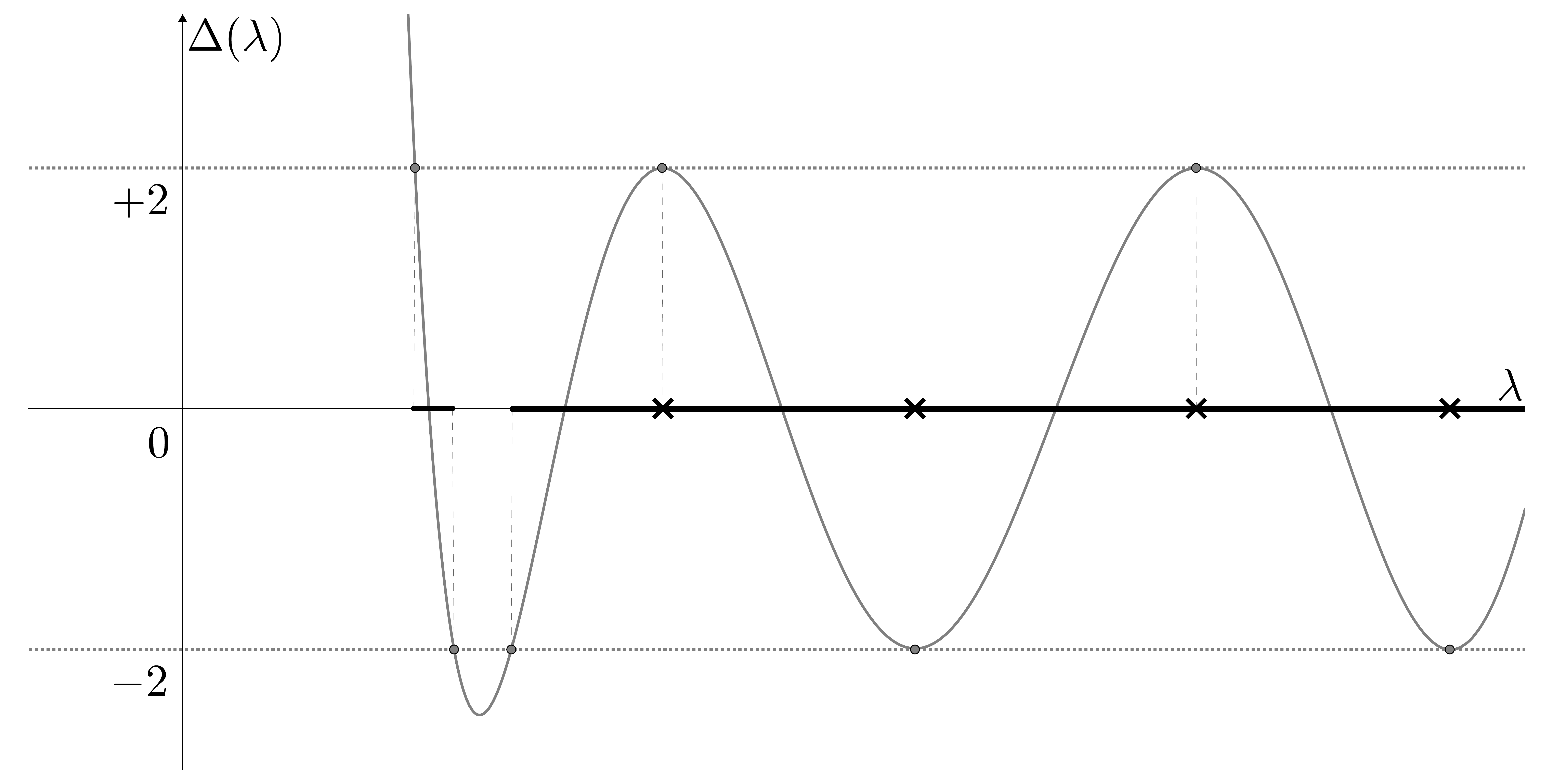}
\caption{The spectrum of the real Lam\'e operator~\eqref{intro_jac_lame} in the $m=1$ case, where~$\Delta(\lambda)$ is the discriminant and crosses denote closed gaps.}\label{Fig:oneGap}
\end{figure}

In the second part of the paper, we consider the Lam\'e operator with a complex-valued periodic potential
\begin{gather}\label{cplxPot}
V(x)= m(m+1)\omega^2\wp(\omega x+z_0),
\end{gather}
where $\omega$ is any half-period of $\wp(x)$, and the only assumption on $z_0 \in \mathbb{C}$ is that the corresponding potential is non-singular.

The spectral theory of Schr\"odinger operators with a complex periodic potential has been studied in a number of papers, see, e.g.,~\cite{Birnir,GW2,RB,T,Weikard} as well as references therein. In particular, Rofe-Beketov showed that the \emph{spectrum} of a Schr\"odinger operator
\begin{gather*}
L = -\frac{{\rm d}^2}{{\rm d}x^2} + u(x),
\end{gather*}
with periodic, regular, but complex-valued potential $u(x)$ can be def\/ined as the set of $\lambda\in\mathbb C$ such that at least one solution of the equation $L\psi=\lambda\psi$ is bounded on the whole real line. Equivalently, the corresponding Floquet multiplier~$\rho(\lambda)$, given by
\begin{gather*}
\psi(x+T,\lambda)=\rho(\lambda)\psi(x,\lambda),
\end{gather*}
where $T$ is a period of $u(x)$, should lie on the unit circle:
\begin{gather*}
|\rho(\lambda)|=1.
\end{gather*}

In the case of the Lam\'e operator, it follows from work of Gesztesy and Weikard \cite{GW1,GW2, Weikard}, who studied the general elliptic f\/inite-gap case, that the spectrum consists of f\/initely many regular analytic arcs, with precisely one arc extending to inf\/inity. In the general case of an operator with a periodic complex-valued potential, the question about the relative positions of the spectral arcs was raised earlier by Pastur and Tkachenko \cite{PT}, who gave an example in which the spectral arcs intersect at interior points. Gesztesy and Weikard~\cite{GW1} observed that more explicit examples are given by the spectrum of the $m=1$ Lam\'e operator for suitable rhombic period lattices, see also Proposition~\ref{Prop:lemn} and Figs.~\ref{Fig:rhombSpec} and~\ref{Fig:lemniSpec} below.

Note that if the shift $z_0\neq \omega_3$ in the Lam\'e operator~(\ref{intro_jac_lame}), we have in general a periodic, regular, but complex-valued potential. However, it is easy to see that the Floquet multipliers do not depend on~$z_0$, so that the spectrum is the same as in the self-adjoint case. Moreover, setting $\omega=\omega_1,\omega_3$ in~\eqref{cplxPot}, we have essentially Ince's result, since the corresponding operator is equivalent to the previous case.

In Section 3, we consider genuinely complex instances of the Lam\'e operator, concentrating on the simplest case $m=1$. Although this case was brief\/ly discussed earlier in a wider context by Gesztesy and Weikard \cite {GW1} and Batchenko and Gesztesy \cite{BG}, we believe that the complete picture is still missing in the literature.

We recall that Hermite's solutions of the corresponding Lam\'e equation
\begin{gather*}
-\frac{{\rm d}^2\psi}{{\rm d}z^2}+2\wp(z)\psi=\lambda \psi, \qquad \lambda=-\wp(k),
\end{gather*}
are given by
\begin{gather}\label{eigenfunction}
\psi(z,k)=\frac{\sigma(z+k)}{\sigma(z)\sigma(k)}\exp(-\zeta(k)z),
\end{gather}
where $k\in\mathbb{C}$ and $\sigma(z)$ and $\zeta(z)$ are the Weierstrass $\sigma$- and $\zeta$-function, see, e.g.,~\cite{WW}. Since the solutions have the Floquet property
\begin{gather*}
\psi(z+2\omega,k)=\exp(2\eta k-2\zeta(k)\omega)\psi(z,k),
\end{gather*}
with $\eta=\zeta(\omega)$, they remain bounded on the line $z=\omega x+z_0$, $x \in \mathbb R$, if and only if
\begin{gather*}
u(k):=\operatorname{Re}[\eta k-\zeta(k)\omega]= 0.
\end{gather*}
By the result of Rofe-Beketov, the corresponding values of $\lambda=-\omega^2\wp(k)$ constitute the spectrum of the Lam\'e operator
\begin{gather}\label{2}
L= -\frac{{\rm d}^2}{{\rm d}x^2}+2\omega^2\wp(x\omega + z_0).
\end{gather}

Thus, the problem is to study the zero level of the real analytic function $u(k)$, $k \in \mathcal{E}^\times=\mathcal E\setminus 0$, where $\mathcal{E}={\mathbb{C}}/\mathcal{L}$ is the corresponding elliptic curve.

We show that $u(k)$ is a Morse function on $\mathcal{E}^\times$ provided the non-degeneracy conditions
\begin{gather*}
\eta+\omega e_j\neq 0,\qquad j=1,2,3,
\end{gather*}
are satisf\/ied. We note that this condition also appeared in Takemura's~\cite{Tak} study of analyticity properties of the eigenvalues of the $m=1$ Lam\'e operator.

Applying Morse theory arguments we show that, under some additional non-singularity assumptions, the spectrum of~\eqref{2} consists of precisely two non-intersecting regular analytic arcs, with one inf\/inite arc asymptotic to the positive half of the real line shifted by $-2\eta\omega$, and the remaining endpoints being $-\omega^2e_j$, $j=1,2,3$ (in agreement with~\cite{GW1} and~\cite{BG}).

The non-degeneracy conditions in general are dif\/f\/icult to control, but in the rhombic case we prove that there exists {\it exactly one exceptional elliptic curve}~$\mathcal E_*$ with the value of the $j$-invariant $j_*\approx 243.797$ (see the top right corner of Fig.~\ref{Fig:rhombkSpec} for the corresponding period rhombus). The corresponding spectrum has a tripod structure, see the top right plot in Fig.~\ref{Fig:rhombSpec}. Note that in this case the spectrum is a union of {\it three} simple regular analytic arcs, which is more than usually expected in the $m=1$ case.

\begin{figure}[t]\centering
\includegraphics[width=0.8\textwidth]{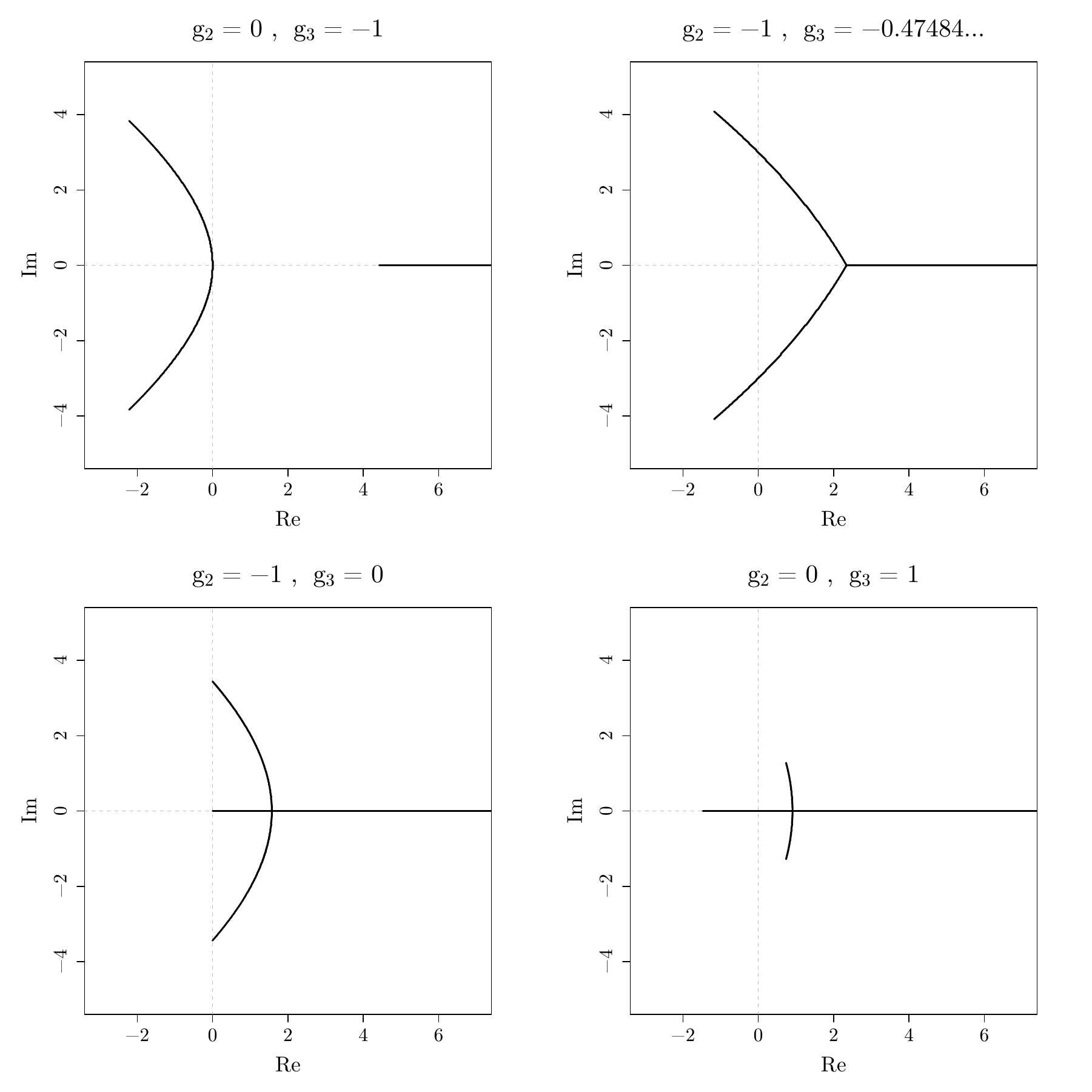}
\caption{Spectra of the complex Lam\'e operator (\ref{2}) in rhombic cases with $\omega=\omega_1$.}\label{Fig:rhombSpec}
\end{figure}

In the rectangular case with $\omega=\omega_2$ we show that all closed spectral gaps are contained in the spectral arc extending to inf\/inity.
Finally, we brief\/ly discuss the spectrum in the rhombic Lam\'e case with $m=2$.

\section{Spectral gaps for the real Lam\'e operator}\label{Sec:realLame}
The real Lam\'e operator (\ref{intro_jac_lame}) can be conveniently rewritten in the Jacobian form as
\begin{gather}\label{jacLamOp}
L_m=-\frac{{\rm d}^2}{{\rm d}x^2}+m(m+1)k^2\operatorname{sn}^2(x,k),
\end{gather}
where $0<k<1$, $m\in\mathbb{N}$ and $\operatorname{sn} (x,k)$ one of the Jacobi elliptic functions, see, e.g.,~\cite{WW}.

It is known after Ince \cite{Ince} and Erd\'elyi~\cite{erdelyi} that its spectrum has a band structure with exactly~$m$ gaps.
We are going to explain now which gaps are open.

\begin{Theorem}\label{lameThm}
The spectrum of the Lam\'e operator~\eqref{jacLamOp} has precisely the first $m$ gaps open.
\end{Theorem}

To prove this we follow the classical approach of Ince~\cite{Ince}, which we explain now following~\cite{MW}.

The Lam\'e potential $m(m+1)k^2\operatorname{sn}^2(x,k)$ is a real smooth periodic function with period $2K$, where
\begin{gather*}
K=\int_0^{\pi/2}\frac{{\rm d}\phi}{\sqrt{1-k^2\sin^2\phi}},
\end{gather*}
so the Lam\'e operator \eqref{jacLamOp} is of Hill type. By the oscillation theorem for Hill operators there exists a sequence of real numbers
\begin{gather}\label{interlace}
\lambda_0<\lambda^\prime_1\leq\lambda^\prime_2<\lambda_1\leq\lambda_2<\cdots<\lambda^\prime_{2n-1}\leq\lambda^\prime_{2n}<\lambda_{2n-1}\leq\lambda_{2n}<\cdots,
\end{gather}
going to inf\/inity, where $\lambda_j$ and $\lambda^\prime_j$ correspond to $2K$-periodic and $2K$-antiperiodic solutions of the Lam\'e equation
\begin{gather}\label{jacLamEq}
-\frac{{\rm d}^2\psi}{{\rm d}x^2}+m(m+1)k^2\operatorname{sn}^2(x,k)\psi=\lambda\psi
\end{gather}
respectively. The intervals $(\lambda_{2n-1},\lambda_{2n})$ and $(\lambda^\prime_{2n-1},\lambda^\prime_{2n})$ are spectral gaps, and the spectrum of the Lam\'e operator \eqref{jacLamOp} is the complement to the union of these gaps and the interval $(-\infty,\lambda_0)$. Such a gap is closed if and only if two linearly independent $2K$-periodic or $2K$-antiperiodic solutions of the Lam\'e equation~\eqref{jacLamEq} coexist for $\lambda=\lambda_{2n-1}=\lambda_{2n}$ or \smash{$\lambda=\lambda^\prime_{2n-1}=\lambda^\prime_{2n}$}, respectively.

So in this language we have to show that {\it all points of coexistence are located in the infinite spectral band.}

To prove this we f\/irst transform the Lam\'e equation~\eqref{jacLamEq} into an equation of Ince's type
\begin{gather}\label{inceeq}
(1+a\cos{2\phi})\frac{{\rm d}^2y}{{\rm d}\phi^2}+b(\sin{2\phi})\frac{{\rm d}y}{{\rm d}\phi}+(c+d\cos{2\phi})y=0,
\end{gather}
where $a$, $b$, $c$, $d$ are real parameters with $|a|<1$ (see~\cite{MW}). More precisely, by the substitutions
\begin{gather*}
\phi=\operatorname{am}x,
\end{gather*}
where the Jacobi amplitude $\operatorname{am}x=\operatorname{am}(k;x)$ is determined by
\begin{gather*}
x=\int_0^{\operatorname{am}x}\frac{{\rm d}\phi}{\sqrt{1-k^2\sin^2\phi}},
\end{gather*}
and $\psi(x)=y(\operatorname{am}x)$, we arrive at \eqref{inceeq} with
\begin{gather}\label{InceLamParam}
a=-b=\frac{k^2}{2-k^2},\qquad c=\frac{2\lambda-m(m+1)k^2}{2-k^2},\qquad d=\frac{m(m+1)k^2}{2-k^2}.
\end{gather}
Note that $x=2K$ corresponds to $\phi=\pi$, which is ref\/lected in the fact that the coef\/f\/icients in~\eqref{inceeq} are periodic with period $\pi$.

If Ince's equation \eqref{inceeq} has two linearly independent $\pi$-periodic or $\pi$-antiperiodic solutions, then we can f\/ind two solutions $y_1$, $y_2$ such that
\begin{gather}
\label{pievenoddsols}
y_1=\sum_{n=0}^\infty A_{2n}\cos2n\phi,\qquad y_2=\sum_{n=1}^\infty B_{2n}\sin2n\phi,
\end{gather}
or
\begin{gather}
\label{2pievenoddsols}
y_1=\sum_{n=0}^\infty A_{2n+1}\cos(2n+1)\phi,\qquad y_2=\sum_{n=0}^\infty B_{2n+1}\sin(2n+1)\phi.
\end{gather}
Substituting the periodic series into \eqref{inceeq} with \eqref{InceLamParam}, we obtain the recurrence relations
\begin{subequations}\label{Anrecurrence}
\begin{gather}
\Lambda_0A_0+P(-1)A_2=0,\\
2P(0)A_0+\Lambda_1A_2+P(-2)A_4=0,\\
P(n-1)A_{2n-2}+\Lambda_nA_{2n}+P(-n-1)A_{2n+2}=0,\qquad n\geq 2
\end{gather}
\end{subequations}
and
\begin{subequations}\label{Bnrecurrence}
\begin{gather}
\Lambda_1B_2+P(-2)B_4=0,\\
P(n-1)B_{2n-2}+\Lambda_nB_{2n}+P(-n-1)B_{2n+2}=0,\qquad n\geq 2
\end{gather}
\end{subequations}
where we have introduced
\begin{gather}\label{P}
P(z)=\frac{m(m+1)k^2}{4}-\frac{z k^2}{2}-z^2k^2,\\
\label{Lambda}
\Lambda_n=\lambda-\frac{m(m+1)k^2}{2}-\left(1-\frac{k^2}{2}\right)4n^2.
\end{gather}
Similarly, the coef\/f\/icients in the antiperiodic series should satisfy the recurrence relations
\begin{subequations}\label{Ansrecurrence}
\begin{gather}
(P^*(0)+\Lambda^*_0)A_1+P^*(-1)A_3=0,\\
P^*(n)A_{2n-1}+\Lambda^*_nA_{2n+1}+P^*(-n-1)A_{2n+3}=0,\qquad n\geq 1
\end{gather}
\end{subequations}
and
\begin{subequations}\label{Bnsrecurrence}
\begin{gather}
(-P^*(0)+\Lambda^*_0)B_1+P^*(-1)B_3=0,\\
P^*(n)B_{2n-1}+\Lambda^*_nB_{2n+1}+P^*(-n-1)B_{2n+3}=0,\qquad n\geq1
\end{gather}
\end{subequations}
with
\begin{gather}\label{Ps}
P^*(z)=\frac{m(m+1)k^2}{4}-\frac{(2z-1)k^2}{4}-\frac{(2z-1)^2k^2}{4},\\
\label{Lambdas}
\Lambda^*_n=\lambda-\frac{m(m+1)k^2}{2}-\left(1-\frac{k^2}{2}\right)(2n+1)^2.
\end{gather}

At this point, we f\/ind it convenient to consider even and odd values of~$m$ separately. Letting
\begin{gather*}
m=2\nu,\qquad \nu\in{\mathbb{Z}}_{\geq 0},
\end{gather*}
we have
\begin{gather*}
P(\nu)=P^*(-\nu)=0,
\end{gather*}
so that the inf\/inite tridiagonal matrices corresponding to the recurrence relations \eqref{Anrecurrence}--\eqref{Bnrecurrence} and \eqref{Ansrecurrence}--\eqref{Bnsrecurrence} become reducible.

In the periodic case, the f\/inite-dimensional parts of these matrices are
\begin{gather}\label{An_finite_mat}
K_{\nu,1}=
\begin{pmatrix}
\Lambda_0 & P(-1) & 0 & \cdots & \cdots & 0 \\
2P(0) & \Lambda_1 & P(-2) & 0 & \cdots & 0 \\
0 & P(1) & \Lambda_2 & P(-3) & \cdots & 0\\
\vdots & \ddots & \ddots & \ddots & \ddots & \vdots\\
0 & \cdots & 0 & P(\nu-2) & \Lambda_{\nu-1} & P(-\nu)\\
0 & \cdots & \cdots & 0 & P(\nu-1) & \Lambda_\nu
\end{pmatrix}
\end{gather}
and
\begin{gather}
\label{Bn_finite_mat}
K_{\nu,2}=
\begin{pmatrix}
\Lambda_1 & P(-2) & 0 & \cdots & 0 \\
P(1) & \Lambda_2 & P(-3) & \cdots & 0 \\
\vdots & \ddots & \ddots & \ddots & \vdots\\
0 & \cdots & P(\nu-2) & \Lambda_{\nu-1} & P(-\nu)\\
0 & \cdots & 0 & P(\nu-1) & \Lambda_\nu
\end{pmatrix},
\end{gather}
which correspond to sequences terminating at $n=\nu$ in the sense that $A_{2n}=B_{2n}=0$ for $n>\nu$.

Introducing the functions
\begin{gather*}
\delta_{\nu,1}(\lambda)=\det K_{\nu,1}(\lambda),\qquad \delta_{\nu,2}(\lambda)=\det K_{\nu,2}(\lambda),
\end{gather*}
we let
\begin{gather*}
S_{\nu,1}=\{\lambda \colon \delta_{\nu,1}(\lambda)=0\},\qquad S_{\nu,2}=\{\lambda \colon \delta_{\nu,2}(\lambda)=0\},
\end{gather*}
be the sets of corresponding eigenvalues of the Lam\'e operator~\eqref{jacLamOp}. Recalling~\eqref{P}, it is clear that all of\/f-diagonal entries of the matrices $K_{\nu,1}$ and $K_{\nu,2}$ are positive. From the general theory of tridiagonal (Jacobi) matrices (see, e.g.,~\cite{Szego}), it follows that $S_{\nu,1}$ and $S_{\nu,2}$ consist of $\nu+1$ respectively $\nu$ real and distinct eigenvalues.

By specializing \cite[Theorem 7.3]{MW} to the Lam\'e case, we f\/ind that $S_\nu:=S_{\nu,1}\cup S_{\nu,2}$ contains the simple part of the periodic spectrum, which we state as the following lemma.

\begin{Lemma} \label{Lemma:perSpec}
Let $m=2\nu$, $\nu\in{\mathbb{Z}}_{\geq 0}$. If $\lambda$ belongs to the periodic spectrum of the Lam\'e opera\-tor~\eqref{jacLamOp} and has multiplicity~$1$, then $\lambda\in S_\nu$.
\end{Lemma}

Turning to the the antiperiodic case, we note that the recurrence relations \eqref{Ansrecurrence}--\eqref{Bnsrecurrence} have no solutions given by terminating sequences. Instead, the relevant f\/inite-dimensional matrices
\begin{gather}\label{Ans_finite_mat}
K_{\nu,1}^*=
\begin{pmatrix}
P^*(0)+\Lambda^*_0 & P(-1) & 0 & \cdots & 0 \\
P^*(1) & \Lambda^*_1 & P^*(-2) & \cdots & 0 \\
\vdots & \ddots & \ddots & \ddots & \vdots\\
0 & \cdots & P^*(\nu-2) & \Lambda^*_{\nu-1} & P^*(-\nu+1)\\
0 & \cdots & 0 & P^*(\nu-1) & \Lambda^*_{\nu-1}
\end{pmatrix}
\end{gather}
and
\begin{gather}\label{Bns_finite_mat}
K_{\nu,2}^*=
\begin{pmatrix}
-P^*(0)+\Lambda^*_0 & P(-1) & 0 & \cdots & 0 \\
P^*(1) & \Lambda^*_1 & P^*(-2) & \cdots & 0 \\
\vdots & \ddots & \ddots & \ddots & \vdots\\
0 & \cdots & P^*(\nu-2) & \Lambda^*_{\nu-1} & P^*(-\nu+1)\\
0 & \cdots & 0 & P^*(\nu-1) & \Lambda^*_{\nu-1}
\end{pmatrix}
\end{gather}
are related to sequences with $A_{2\nu-1}, B_{2\nu-1}\neq 0$. Indeed, such solutions of the recurrence relations \eqref{Ansrecurrence}--\eqref{Bnsrecurrence} exist only if the matrices $K_{\nu,1}^*$, $K_{\nu,2}^*$ are singular.

Letting
\begin{gather*}
\delta_{\nu,1}^*(\lambda)=\det K_{\nu,1}^*(\lambda),\qquad \delta_{\nu,2}^*(\lambda)=\det K_{\nu,2}^*(\lambda),
\end{gather*}
we introduce the sets
\begin{gather*}
S_{\nu,1}^*=\{\lambda \colon \delta_{\nu,1}^*(\lambda)=0\},\qquad S_{\nu,2}^*=\{\lambda \colon \delta_{\nu,2}^*(\lambda)=0\}.
\end{gather*}
Just as in the periodic case, we see that each of the two sets $S_{\nu,j}^*$ consists of $\nu$ real and distinct eigenvalues. Moreover, we infer from \cite[Theorem~7.4]{MW} that $S^*_\nu:=S^*_{\nu,1}\cup S^*_{\nu,2}$ contains the simple part of the antiperiodic spectrum.

\begin{Lemma}\label{Lemma:antperSpec}
Let $m=2\nu$, $\nu\in{\mathbb{Z}}_{\geq 0}$. If $\lambda$ belongs to the antiperiodic spectrum of the Lam\'e operator~\eqref{jacLamOp} and has multiplicity $1$, then $\lambda\in S^*_\nu$.
\end{Lemma}

Since the spectrum of the Lam\'e operator~\eqref{jacLamOp} has precisely $m=2\nu$ gaps, the ends of the gaps must be given by the $m$ largest eigenvalues in $S_\nu$ and all $m$ eigenvalues in~$S^*_{\nu}$. (Recall that $(-\infty,\lambda_0)$ is not counted as a gap.)

We are now ready to complete the proof of Theorem~\ref{lameThm} for even $m$. From \eqref{P}--\eqref{Lambda} and \eqref{An_finite_mat}--\eqref{Bn_finite_mat} it easily follows that
\begin{gather*}
\delta_{\nu,j}(k)=\prod_{n=j-1}^\nu \big(\lambda-(2n)^2\big)+\mathcal{O}\big(k^2\big),\qquad k\to 0,
\end{gather*}
where $j=1,2$. Similarly, \eqref{Ps}--\eqref{Lambdas} and \eqref{Ans_finite_mat}--\eqref{Bns_finite_mat} readily yield
\begin{gather*}
\delta^*_{\nu,j}(k)=\prod_{n=0}^{\nu-1} \big(\lambda-(2n+1)^2\big)+\mathcal{O}\big(k^2\big),\qquad k\to 0.
\end{gather*}
This implies that the $m$ gaps close at the points $\lambda=j^2$, $j=1,\ldots,m$, as $k\to 0$. On the other hand, taking $k\to 0$ in \eqref{jacLamEq}, we obtain the free Schr\"odinger equation $-\psi^{\prime\prime}=\lambda\psi$, and (excluding the constant solution $\psi=1$ with $\lambda=0$) the $m$ smallest values of $\lambda$ for which it has $\pi$-periodic or $\pi$-antiperiodic solutions are precisely $\lambda=j^2$, $j=1,\ldots,m$. Using that the eigenvalues $\lambda=\lambda(k)\in S_\nu\cup S^*_\nu$ have multiplicity one, a standard argument shows that they depend analytically on $k$, see e.g.~Chapter XII in \cite{ReedSimon2}. Hence there exists $k_1\in(0,1)$ such that all points of coexistence are indeed contained in the inf\/inite spectral band for $k\in(0,k_1)$.

There remains to show that neither of the following two scenarios can occur for $k\geq k_1$: f\/irst, a~point of coexistence crosses one or more open gaps; and, second, an open gap closes while a~point of coexistence turns into an open gap. The f\/irst scenario is immediately ruled out by the interlacing property~\eqref{interlace}, and the second by the analyticity of the eigenvalues.

This concludes the proof in the case of even $m$. It is straightforward to adapt the arguments above to handle the case of odd $m$. In particular, the roles of the periodic and antiperiodic solutions will then be interchanged, and analogues of Lemmas~\ref{Lemma:perSpec}--\ref{Lemma:antperSpec} can once more be obtained from \cite[Theorems~7.3--7.4]{MW}.

\begin{Example}
Consider the simplest nontrivial even case $m=2$, which corresponds to the Lam\'e operator
\begin{gather}\label{jacm2LamOp}
L_2=-\frac{{\rm d}^2}{{\rm d}x^2}+6k^2\operatorname{sn}^2(x,k).
\end{gather}
The edges of its two spectral gaps $(\lambda^\prime_1,\lambda^\prime_2)$ and $(\lambda_1,\lambda_2)$, which correspond to the f\/irst two eigenvalues $\lambda_1^\prime$, $\lambda^\prime_2$ from its antiperiodic spectrum and the second and third eigenvalues~$\lambda_1$,~$\lambda_2$ from its periodic spectrum, are easily computed explicitly.

Indeed, observing that
\begin{gather*}
\delta^*_{1,1}=P^*(0)+\Lambda^*_0,\qquad \delta^*_{1,2}=-P^*(0)+\Lambda^*_0,
\end{gather*}
we f\/ind that the edges of the f\/irst gap are given by
\begin{gather}\label{eigenVal1}
\lambda^\prime_1=1+k^2,\qquad \lambda^\prime_2=1+4k^2,
\end{gather}
cf.~\eqref{Ps}--\eqref{Lambdas}. To determine the edges of the second gap, we should f\/ind the roots $\lambda$ of the polynomials
\begin{gather*}
\delta_{1,1}=
\begin{vmatrix}
 \Lambda_0& P_{-1}\\
 2P_0 & \Lambda_1\\
 \end{vmatrix},\qquad
\delta_{1,2}=\Lambda_1.
\end{gather*}
Recalling \eqref{P}--\eqref{Lambda}, we deduce that the two roots of the former polynomial are
\begin{gather}\label{eigenVal2}
\lambda_0=2\big(1+k^2-\sqrt{1-k^2+k^4}\big),\qquad \lambda_2=2\big(1+k^2+\sqrt{1-k^2+k^4}\big),
\end{gather}
whereas the latter has the single root
\begin{gather}\label{eigenVal3}
\lambda_1=4+k^2.
\end{gather}

The relevant eigenvalues are plotted together with the ground state eigenvalue from the periodic spectrum in Fig.~\ref{Fig:m=2EigenVals}.

In the limit $k\to 1$ the real period goes to inf\/inity and the Lam\'e potential
\begin{gather*}
6k^2\operatorname{sn}^2(x,k)\to 6-6\operatorname{sech}^2 x.
\end{gather*}
The potential $u(x)=-6\operatorname{sech}^2 x$ decays exponentially at inf\/inity and is a special example of a~$2$-soliton potential, see, e.g.,~\cite{DJ}. The corresponding operator
\begin{gather*}
H_2=-\frac{{\rm d}^2}{{\rm d}x^2}-6\operatorname{sech}^2 x
\end{gather*}
is known to be ref\/lectionless. Its spectrum has continuous part $[0,\infty)$ and discrete part consisting of the two eigenvalues $-1$ and $-4$. Note that for $k\to 1$ the formulae \eqref{eigenVal1}--\eqref{eigenVal3} show that the two f\/inite spectral bands degenerate to the eigenvalues $\lambda_0=\lambda^\prime_1=2$ and $\lambda^\prime_2=\lambda_1=5$, while $\lambda_2\to 6$, in agreement with Fig.~\ref{Fig:m=2EigenVals}.

\begin{figure}[t]\centering
\includegraphics[width=0.75\textwidth]{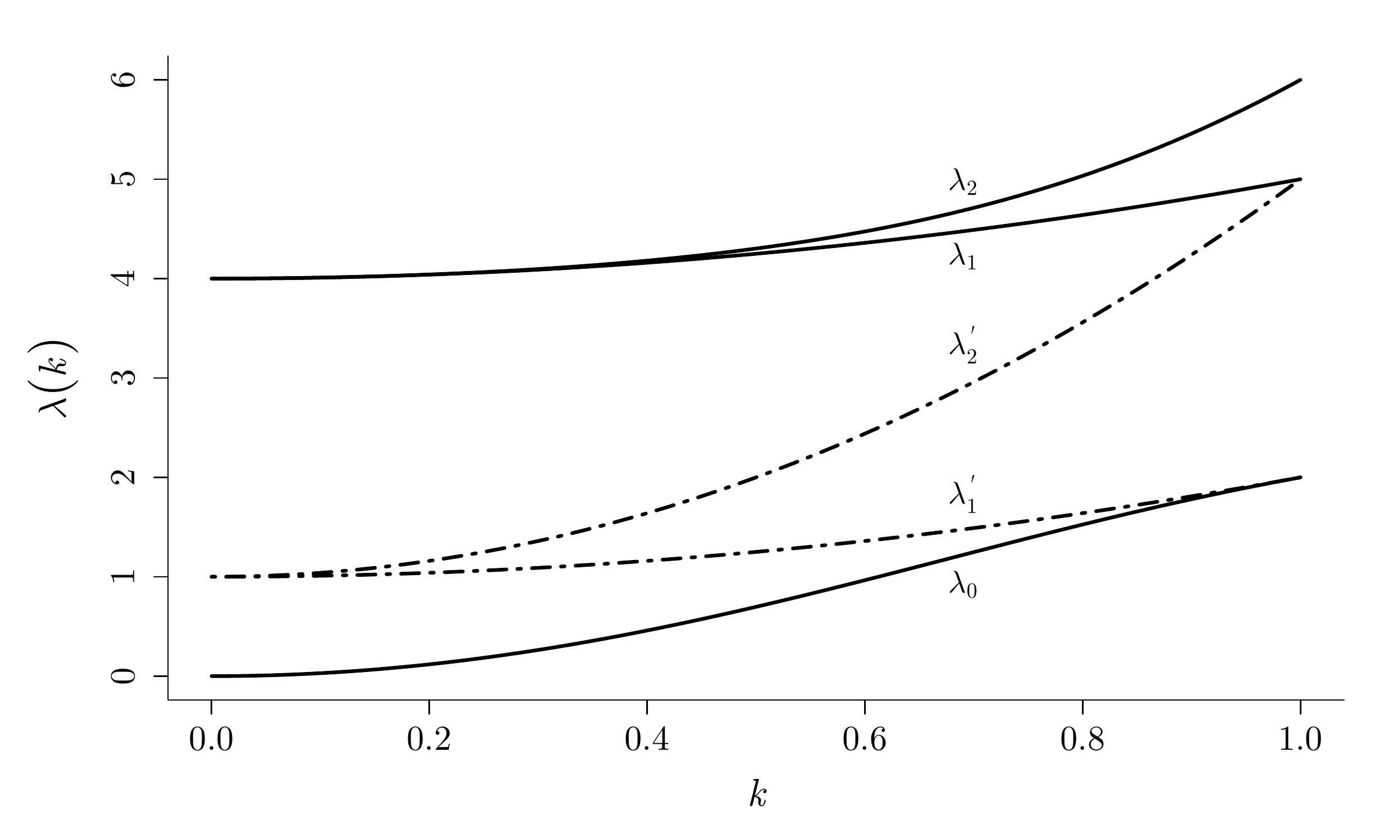}
\caption{First f\/ive eigenvalues from the periodic and antiperiodic spectra of the Lam\'e operator \eqref{jacm2LamOp}.}\label{Fig:m=2EigenVals}
\end{figure}
\end{Example}

\section{The spectrum of the complex Lam\'e operator}\label{Sec:cplxLame}
Let $\mathcal{E}={\mathbb{C}}/\mathcal{L}$ be a general elliptic curve, where $\mathcal{L}$ is a period lattice. To begin with, we do not impose any reality conditions. Let $\wp(z)$ be the corresponding Weierstrass' elliptic function,
\begin{gather*}
\wp(z+\Omega)=\wp(z),\qquad \Omega\in\mathcal{L},
\end{gather*}
which satisf\/ies the equation
\begin{gather*}
(\wp^\prime)^2=4(\wp-e_1)(\wp-e_2)(\wp-e_3),
\end{gather*}
see, e.g.,~\cite{WW}. It has poles of second order at all lattice points, which are the only singularities of~$\wp(z)$.

Let $\omega$ be a half-period: $2\omega \in \mathcal L$, and consider the {\it complex Lam\'e operator} $L=L(\mathcal{E},\omega,m,z_0)$ in $L^2({\mathbb{R}})$ given by
\begin{gather}\label{cplxLameOp}
L=-\frac{{\rm d}^2}{{\rm d}x^2}+m(m+1)\omega^2\wp(\omega x+z_0),
\end{gather}
with $m\in\mathbb N$ and $z_0\in\mathbb C$ chosen such that the line $z=\omega x+z_0$, $x\in{\mathbb{R}}$, does not contain any points of $\mathcal{L}$. Note that the potential $m(m+1)\omega^2\wp(\omega x+z_0)$ is regular and periodic with period~$2$, but is in general complex-valued.

From the work of Rofe-Beketov \cite{RB} on the spectral theory of non-self-adjoint dif\/ferential operators in $L^2({\mathbb{R}})$ with complex-valued periodic coef\/f\/icients, it follows that the spectrum of $L$ can be characterised as the set of $\lambda \in \mathbb C$ such that the associated Lam\'e equation
\begin{gather}\label{realLameEq}
-\frac{{\rm d}^2\phi}{{\rm d}x^2}+m(m+1)\omega^2\wp(\omega x+z_0)\phi=\lambda\phi,\qquad x\in{\mathbb{R}},
\end{gather}
has a non-zero bounded solution. Equivalently, the corresponding Floquet multiplier $\mu(\lambda)$, def\/ined by
\begin{gather*}
\phi(x+2,\lambda)=\mu(\lambda)\phi(x,\lambda),
\end{gather*}
must have absolute value $|\mu(\lambda)|=1$.

\begin{Proposition}\label{Prop:cplxLame}
The spectrum of the complex Lam\'e operator \eqref{cplxLameOp} does not depend on the value of $z_0$ and is determined by the condition~\eqref{spect} below.
\end{Proposition}

\begin{proof}
We will use the classical results about the Lam\'e equation
\begin{gather}\label{cplxLameEq}
	-\frac{{\rm d}^2\psi}{{\rm d}z^2} + m(m+1)\wp(z)\psi=\lambda\psi
\end{gather}
in the complex domain $z\in{\mathbb{C}}$. Its solutions were described by Hermite \cite{Hermite} in the following way, see, e.g., \cite[Section 23.7]{WW}. First, one notes that the product $X$ of any pair of solutions of \eqref{cplxLameEq} satisf\/ies the third order equation
\begin{gather*}
-\frac{{\rm d}^3 X}{{\rm d}z^3}+4\big(m(m+1)\wp(z)-\lambda\big)\frac{{\rm d} X}{{\rm d}z}+2m(m+1)\wp^\prime(z)=0.
\end{gather*}
Changing variable to
\begin{gather*}
\xi=\wp(z)
\end{gather*}
yields the algebraic form of this equation:
\begin{gather*}
-4(\xi-e_1)(\xi-e_2)(\xi-e_3)\frac{{\rm d}^3X}{{\rm d}\xi^3}-3\big(6\xi^2-g_2/2\big)\frac{{\rm d}^2X}{{\rm d}\xi^2}\\
\qquad{} +4((m^2+m-3)\xi-\lambda)\frac{{\rm d}X}{{\rm d}\xi}-2m(m+1)X=0.
\end{gather*}
Making a power series ansatz in $\xi-e_2$ (say), it is straightforward to verify that the Lam\'e equation~\eqref{cplxLameEq} has two solutions~$\psi^\pm$ whose product~$X$ is a polynomial in $\xi-e_2=\wp(z)-e_2$ of the form
\begin{gather*}
X(\xi)=(\xi-e_2)^m+\sum_{r=1}^m c_r(\xi-e_2)^{m-r},
\end{gather*}
with some coef\/f\/icients $c_r\in\mathbb C$. This polynomial can be written in the factorised form
\begin{gather*}
X(\xi)=\prod_{j=1}^m(\xi-\wp(k_j)),
\end{gather*}
where the $k_j=k_j(\lambda)\in{\mathbb{C}}$ are determined by this relation up to permutation and up to a sign. Assume now that $\psi^\pm$ are linearly independent, so that their Wronskian
\begin{gather}\label{WronEq}
W\big(\psi^+,\psi^-\big)\equiv C\neq 0.
\end{gather}
(Otherwise $\psi^\pm$ are one of the $2m+1$ Lam\'e functions.) The sign of $k_j$ can then be f\/ixed by requiring
\begin{gather*}
\frac{{\rm d}X}{{\rm d}\xi}\Bigg\arrowvert_{\xi=\wp(k_j)}=\frac{C}{\wp^\prime(k_j)}.
\end{gather*}
Finally, by solving \eqref{WronEq} and ${\rm d}(\psi^+\psi^-)/{\rm d}z={\rm d}X/{\rm d}z$ for ${\rm d}\log\psi^\pm/{\rm d}z$ and integrating, one f\/inds that~\eqref{cplxLameEq} is satisf\/ied by the two functions
\begin{gather}\label{gen_m_eigen}
\psi^\pm=\prod_{j=1}^m\left(\frac{\sigma(z\pm k_j)}{\sigma(z)\sigma(k_j)}\right){\rm e}^{\mp\sum\limits_{j=1}^m\zeta(k_j)z}.
\end{gather}

Restricting \eqref{gen_m_eigen} to the line $z=\omega x+z_0$ and using the quasi-periodicity of the $\sigma$- and $\zeta$-function we see that \eqref{realLameEq} has the Floquet solutions with Floquet multiplier $\mu = {\rm e}^{2f_m}$,
\begin{gather*}
f_m(\lambda)=\zeta(\omega)\sum_{j=1}^m k_j - \omega\sum_{j=1}^m\zeta(k_j).
\end{gather*}
Since $f_m$ is manifestly independent of $z_0$, it follows that the spectrum of the complex Lam\'e operator \eqref{cplxLameOp}, determined by the condition
\begin{gather}\label{spect}
 \operatorname{Re} [f_m(\lambda)]=0,
\end{gather}
is indeed independent of the value of $z_0$.
\end{proof}

Unfortunately, the condition \eqref{spect} is very dif\/f\/icult to analyse in general, so we will now focus on the simplest non-trivial case $m=1$. Hermite's solutions of (\ref{cplxLameEq}) have the form (\ref{eigenfunction}):
\begin{gather*}
\psi(z,k)=\frac{\sigma(z+k)}{\sigma(z)\sigma(k)}\exp(-\zeta(k)z),
\end{gather*}
with parameter $k$ related to $\lambda$ by $\lambda=-\wp(k)$. Restriction to the line $z=\omega x+z_0$ gives the solutions to \eqref{realLameEq} with
\begin{gather*}
\lambda=-\omega^2 \wp(k)
\end{gather*}
and Floquet multiplier given by
\begin{gather*}
f(k)=\eta k-\omega\zeta(k).
\end{gather*}
In particular, the corresponding solutions are bounded for $x\in{\mathbb{R}}$ if and only if
\begin{gather}\label{uzero}
u(k):= \operatorname{Re}[f(k)]=0.
\end{gather}

To study the solutions of this transcendental equation we will use elementary Morse theory arguments, see, e.g.,~\cite{Mil}.

\begin{Lemma}\label{Lemma:Morse}
Assuming that
\begin{gather}\label{nonDeg}
\eta+\omega e_j\neq 0,\qquad j=1,2,3,
\end{gather}
$u(k)$ is a Morse function on $\mathcal{E}^\times$, with critical points $\pm k^*$ given by
\begin{gather*}
\eta + \omega\wp(k^*)=0.
\end{gather*}
Moreover, the index of $u(k)$ at $k=\pm k^*$ is equal to one.
\end{Lemma}

\begin{proof}
Given that $\omega$ is a half-period, there exist $\nu_1,\nu_3\in{\mathbb{Z}}$ such that
\begin{gather*}
\omega = \nu_1\omega_1 + \nu_3\omega_3.
\end{gather*}
By the quasi-periodicity of $\zeta(z)$, we have
\begin{gather*}
u(k+2\omega_j) - f(k) = \nu_l\operatorname{Re}[\eta_l\omega_j-\eta_j\omega_l],\qquad l\neq j.
\end{gather*}
Hence the Legendre relation $\eta_1\omega_3-\eta_3\omega_1=i\pi/2$ ensures that the right-hand side equals zero, so that $u(k)$ is a well-def\/ined function on $\mathcal{E}^\times$.

Writing $k=s+it$, the critical points of $u(k)$ are the solutions of $u_s=u_t=0$, which are equivalent to
\begin{gather*}
\frac{{\rm d}f(k)}{{\rm d}k}=\frac{{\rm d}(\eta k-\omega\zeta(k))}{{\rm d}k} = \eta + \omega\wp(k) = 0.
\end{gather*}
Since $\wp(k)$ is an even function of order two, this equation has precisely two solutions $k=\pm k^*$. Letting
\begin{gather*}
v(k)=\operatorname{Im}[f(k)],
\end{gather*}
we infer from the Cauchy--Riemann equations that the Hessian matrix
\begin{gather*}
H(u) =
\begin{bmatrix}
u_{ss} & v_{tt}\\
v_{tt} & -u_{ss}
\end{bmatrix},
\end{gather*}
which has eigenvalues
\begin{gather*}
\lambda_\pm=\pm\sqrt{u_{ss}^2+v_{tt}^2}.
\end{gather*}
We recall that $u(k)$ is a Morse function if $H(u)$ is non-singular at all of its critical points, with the index being equal to the number of negative eigenvalues. Clearly, $H(u)$ is singular if and only if $u_{ss}=v_{tt}=0$ or, equivalently,
\begin{gather*}
\frac{{\rm d}^2f(k)}{{\rm d}k^2}=\frac{{\rm d}^2(\eta k-\omega\zeta(k))}{{\rm d}k^2} = \omega\wp^\prime(k) = 0.
\end{gather*}
Recalling that the only zeros of $\wp^\prime(k)$ on $\mathcal{E}^\times$ are $k=\omega_j$, $j=1,2,3$, we see that $H(u)$ is singular only if $\eta+\omega e_j=0$ for some $j=1,2,3$, which is excluded by the assumption~\eqref{nonDeg}.
\end{proof}

Using Morse theory, we can thus determine the topological nature of the zero set \eqref{uzero}. We assume that~\eqref{nonDeg} holds true, so that the critical points of $u(k)$ are non-degenerate, and that the level set $u(k)=0$ is non-singular in the sense that
\begin{gather}\label{nonSing}
u(k^*)\neq 0.
\end{gather}
We note that the lemniscatic case yields an example of a singular level set and the rhombic cases contain singular examples as well as a degenerate level set, see Figs.~\ref{Fig:lemniSpec} and~\ref{Fig:rhombSpec}.

\begin{Proposition}\label{Prop:kSpec}
Under our non-degeneracy and non-singularity assumptions \eqref{nonDeg} and \eqref{nonSing}, the closure $\bar{Z}_u\subset\mathcal{E}$ of the zero set
\begin{gather*}
Z_u=\big\{k\in\mathcal{E}^\times \colon u(k)=0\big\}
\end{gather*}
consists of two simple closed curves in $\mathcal{E}$ representing the same $($non-trivial$)$ homology class.
\end{Proposition}

\begin{proof}
By our non-singularity assumption, we can f\/ix the sign of $k^*$ such that
\begin{gather*}
c^*:=u(k^*)>0.
\end{gather*}
For $c\in{\mathbb{R}}$, we let $\mathcal{E}_c=\overline{\mathcal{E}_c^\times}\subset\mathcal{E}$ be the closure of the set
\begin{gather*}
\mathcal{E}_c^\times:=\big\{k\in\mathcal{E}^\times \colon -\infty<u(k)\leq c\big\}
\end{gather*}
(where, as we shall see below, taking the closure amounts to adding the point $k=0$).

First, we prove that each $\mathcal{E}_c$ with $c<-c^*$ is dif\/feomorphic to a disc. Since the meromorphic function $f(k)=\eta k-\omega\zeta(k)$ is odd and has a simple pole at $k=0$ with residue $-\omega$, we can f\/ind neighbourhoods $U,V\ni 0$ and a biholomorphic mapping $g\colon U\to V$ such that
\begin{enumerate}\itemsep=0pt
\item[1)] $g(0)=0$ and $g^\prime(0)=1$,
\item[2)] $g(-k)=-g(k)$, $k\in U$,
\item[3)] $f\big(g^{-1}(z)\big) = -\omega/z$, $z\in V$.
\end{enumerate}
Letting
\begin{gather*}
\omega=\alpha+i\beta,\qquad z=x+iy,
\end{gather*}
the set $\operatorname{Re}[-\omega/z]\leq c$, $z\neq 0$, is given by
\begin{gather}\label{disc}
\left(x+\frac{\alpha}{2c}\right)^2 + \left(y+\frac{\beta}{2c}\right)^2\leq \frac{\alpha^2+\beta^2}{4c^2},
\end{gather}
where $(x,y)\neq (0,0)$. By choosing $c^\prime<-c^*$ suf\/f\/iciently small, we can ensure that the disc~\eqref{disc} is contained within $V$ and therefore that it is dif\/feomorphic to $\mathcal{E}_{c^\prime}$. For $c^\prime<c<-c^*$, we infer from \cite[Theorem~3.1]{Mil} that $\mathcal{E}_c$ is dif\/feomorphic to $\mathcal{E}_{c^\prime}$, and the assertion follows. (To be precise, the theorem does not apply as it stands, since $u(k)$ is not def\/ined at $k=0$. However, $\langle \operatorname{grad} u,\operatorname{grad} u\rangle ^{-1}\operatorname{grad} u$ extends to a smooth vector f\/ield on a neighbourhood of $\mathcal{E}_c$, $c<-c^*$, and the construction of the dif\/feomorphisms $\varphi_t\colon \mathcal{E}\to\mathcal{E}$ is readily adapted to the present case.)

Letting $0<\epsilon<c^*$, the region $u^{-1}([-c^*-\epsilon,-c^*+\epsilon])$ contains no critical point of $u$ other than $-k^*$. Since the index of $u(k)$ at $k=-k^*$ is one, it follows that $\mathcal{E}_{-c^*+\epsilon}$ is dif\/feomorphic to a~disc with a thickened $1$-cell attached, cf.~\cite[Theorem~3.2]{Mil}. Moreover, increasing~$c$ from $-c^*+\epsilon$ to $0$ does not alter the dif\/feomorphism type of~$\mathcal{E}_c$. Considering the boundary of~$\mathcal{E}_0$, we thus arrive at the statement.
\end{proof}

This leads us to the following result, which can also be extracted from \cite[Theorem 1.1 and Example C.1]{BG}.

\begin{Proposition}\label{Thm:cplxLameThm}
Under our assumptions of non-degeneracy \eqref{nonDeg} and non-singularity \eqref{nonSing}, the spectrum of the complex Lam\'e opera\-tor~\eqref{cplxLameOp} with $m=1$ consists of two regular analytic arcs. Moreover, precisely one arc extends to infinity and the remaining endpoints are $-\omega^2e_j$, $j=1,2,3$.
\end{Proposition}

\begin{proof}
By Proposition \ref{Prop:kSpec}, there exist simple closed curves $\gamma_1$, $\gamma_2$ such that
\begin{gather*}
\bar{Z}_u=\gamma_1\cup\gamma_2.
\end{gather*}
For each $j=1,2,3$, we deduce from the Legendre relation that $\omega_j\in Z_u$. In addition, $\bar{Z}_u$ contains $k=0$ and is invariant under the ref\/lection $k\to -k$, which implies that each curve $\gamma_j$ contains precisely two f\/ixed points. Since the quotient map
\begin{gather*}
\mathcal{E}\to\mathcal{E}/(k\to -k)=\mathbb{CP}^1
\end{gather*}
is given by $k\mapsto\wp(k)$, the image of $\bar{Z}_u$ in $\mathbb{CP}^1$ under the map $k\mapsto\lambda=-\omega^2\wp(k)$ consists of two regular analytic arcs connecting two pairs of f\/ixed points.
\end{proof}

The next proposition, which can also be extracted from \cite{BG}, describes the asymptotic behaviour of the inf\/inite spectral arc.

\begin{Proposition}\label{Prop:asymp}
For $m=1$, the spectral arc of the complex Lam\'e operator \eqref{cplxLameOp} that extends to infinity is asymptotic to the positive real line shifted by $-2\eta\omega$.
\end{Proposition}

\begin{proof}
Recalling from the proof of Proposition \ref{Prop:kSpec} the biholomorphic mapping $g\colon U\to V$, we observe that any $z\in V$ satisfying the spectral condition $u\big(g^{-1}(z)\big)=\operatorname{Re}[-\omega/z]=0$ is of the form
\begin{gather*}
z=i\omega s,\qquad s\neq 0.
\end{gather*}
By Properties (1)--(3) of $g$, we have
\begin{gather*}
g^{-1}(z)=z-\frac{\eta}{\omega}z^3+\mathcal{O}\big(z^5\big).
\end{gather*}
This implies
\begin{gather*}
\lambda = -\omega^2\wp\big(g^{-1}(i\omega s)\big)= \frac{1}{s^2} - 2\eta\omega + \mathcal{O}\big(s^2\big),\qquad s\to 0,
\end{gather*}
from which the assertion clearly follows.
\end{proof}

For the remainder of this section, we restrict attention to {\it real} elliptic curves $\mathcal{E}$ given by
\begin{gather*}
y^2=4x^3-g_2x-g_3,\qquad g_2,g_3\in\mathbb{R}.
\end{gather*}
There are two types of such curves, depending on the sign of the discriminant
\begin{gather*}
\Delta=g_2^3-27g_3^2.
\end{gather*}

If $\Delta>0$, the roots $e_1$, $e_2$, $e_3$ of the polynomial $p(x)=4x^3-g_2x-g_3$ are all real, the corresponding real curve consists of two ovals ($M$-curve) and the period lattice~$\mathcal{L}$ is rectangular.

If $\Delta<0$, we have one real root $e_1$ and a pair $e_2$, $e_3$ of complex conjugate roots, $e_3=\overline{e_2}$. The corresponding real curve consists of one oval and the period lattice~$\mathcal{L}$ is rhombic.

We f\/irst consider the rectangular case, with basic half-periods
\begin{gather}
\label{omega13Rect}
\omega_1\in (0,\infty),\qquad \omega_3\in i(0,\infty).
\end{gather}
The simplest instance of the corresponding Lam\'e operator \eqref{cplxLameOp} that is truly complex is obtained by choosing
\begin{gather*}
\omega=\omega_2=\omega_1+\omega_3.
\end{gather*}
For this choice of $\omega$, we have used the software {\it R} to numerically compute the spectrum of~\eqref{cplxLameOp} for $g_3=1$ and four dif\/ferent values of~$g_2$. The results are displayed in Fig.~\ref{Fig:rectSpec}.

\begin{figure}[t]\centering
\includegraphics[width=0.75\textwidth]{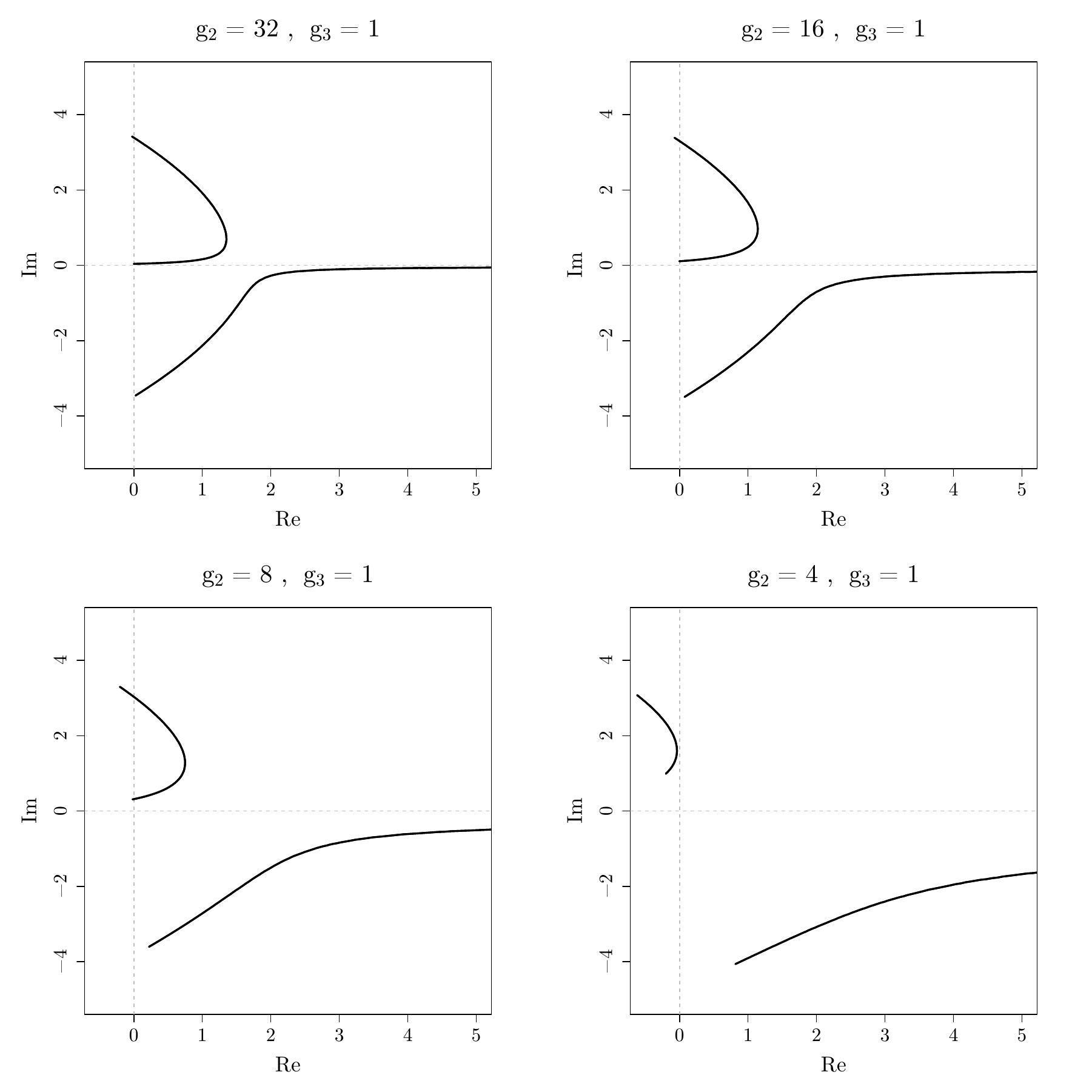}
\caption{Spectra of the complex Lam\'e operator \eqref{cplxLameOp} for rectangular period lattices ${\mathcal L}$, $m=1$ and $\omega=\omega_2$.}\label{Fig:rectSpec}
\end{figure}

We note that the (anti-)periodic solutions of \eqref{realLameEq} are given by $k\in Z_u$ satisfying
\begin{gather}\label{v}
v(k):=\operatorname{Im}[f(k)]=p\frac{\pi}{2},\qquad p\in\mathbb{Z}.
\end{gather}
By the Legendre relation, $v(\omega_3)=-v(\omega_1)=\pi/2$ and $v(\omega_2)=0$. Since $\lambda=-\omega^2\wp(\omega_j)$, $j=1,2,3$, at the endpoints of the spectral arcs, the remaining values of~$p$ must correspond to points contained in one of the two spectral arcs. Following standard terminology for the real Lam\'e operator, we call such a point in the spectrum a {\it closed spectral gap}. In analogy with Theorem~\ref{lameThm}, we have the following result.

\begin{Theorem}\label{Thm:rect}
For $m=1$, $\omega=\omega_2$ and $\omega_1$, $\omega_3$ of the form~\eqref{omega13Rect}, all closed spectral gaps of the complex Lam\'e operator~\eqref{cplxLameOp} are located on the spectral arc extending to infinity.
\end{Theorem}

Our proof of this result relies on a detailed analysis of the lemniscatic case. The corresponding set $\bar{Z}_u$ and spectrum of \eqref{cplxLameOp} are shown in Fig.~\ref{Fig:lemniSpec}. The curve $k=(1-i)\omega_1s$, $-1\leq s\leq 1$, and spectral arc $[0,\infty)$ are deduced analytically in the proof of Proposition~\ref{Prop:lemn} below, whereas the remaining curve and spectral arc were obtained numerically using the software~{\it R}.

\begin{figure}[t]\centering
\includegraphics[width=0.75\textwidth]{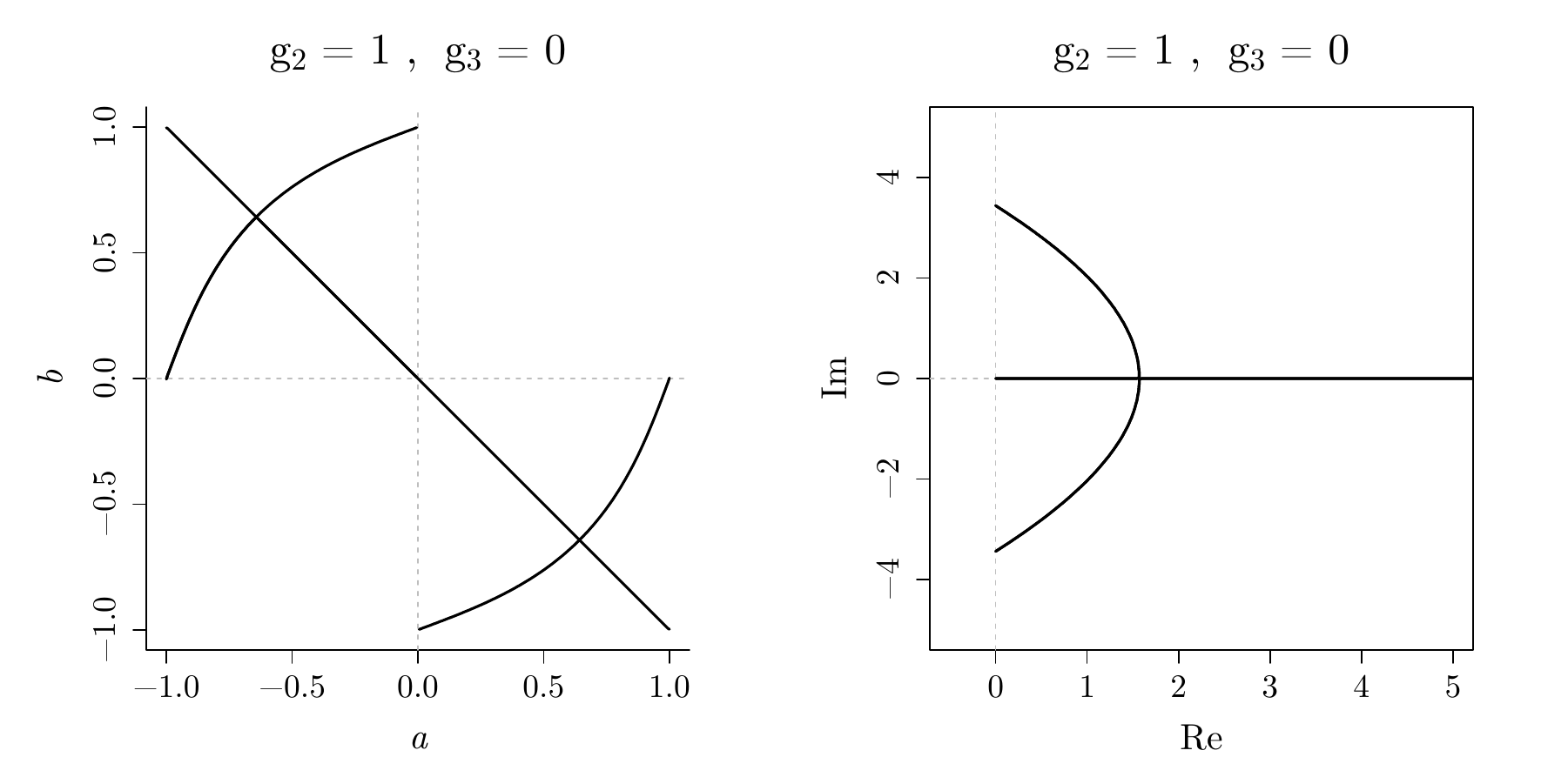}
\caption{The set $\bar{Z}_u$, where $k=a\omega_1+b\omega_3$, and the spectrum of the complex Lam\'e operator~\eqref{cplxLameOp} in the lemniscatic case with $m=1$ and~$\omega=\omega_2$.} \label{Fig:lemniSpec}
\end{figure}

\begin{Proposition}\label{Prop:lemn}
Fix $\omega_1\in (0,\infty)$ and let $\omega_3=i\omega_1$ and $\omega=(1+i)\omega_1$. Then the spectrum of the complex Lam\'e operator~\eqref{cplxLameOp} consists of two analytic arcs, intersecting at the interior point $\lambda^*=\eta_2\omega_2$. The spectrum is invariant under complex conjugation and the infinite spectral arc coincides with $[0,\infty)$. Moreover, the tangent lines to the finite spectral arc at its endpoints
\begin{gather}\label{endpts}
\lambda_\pm=\pm i\frac{\Gamma^4(1/4)}{16\pi}
\end{gather}
are
\begin{gather}\label{tngnts}
l_\pm(t)=\lambda_\pm+\left(\frac{\Gamma^4(1/4)}{8\pi^2}\pm i\right)^2t,\qquad t\in\mathbb{R}.
\end{gather}
\end{Proposition}

\begin{proof}
For the lemniscatic $\zeta$- and $\wp$-function, we recall the identities
\begin{gather}\label{ids}
\zeta(iz)=-i\zeta(z),\qquad \wp(iz)=-\wp(z),
\end{gather}
and special values
\begin{gather}
\label{vals}
\eta_1=i\eta_3=\frac{\pi}{4\omega_1},\qquad e_1=-e_3=\frac{\Gamma^4(1/4)}{32\pi\omega_1^2},\qquad e_2=0,
\end{gather}
see, e.g., \cite[Section~23.5(iii)]{DLMF}.

We note that the spectrum is given by
\begin{gather*}
\lambda=-\omega_2^2\wp(k)=-2\omega_1^2i\wp(k),\qquad k\in Z_u.
\end{gather*}
From \eqref{ids}--\eqref{vals}, we infer
\begin{gather*}
\bar{\lambda}=-2\omega_1^2i\wp(i\bar{k})
\end{gather*}
and
\begin{gather*}
u(k)=\operatorname{Re}\left[\frac{\pi}{4\omega_1}(1-i)k-(1+i)\omega_1\zeta(k)\right]=u(i\bar{k}),
\end{gather*}
which clearly implies invariance of the spectrum under complex conjugation. Using the f\/irst identity in \eqref{ids}, we f\/ind that $\overline{f((1-i)s)}=-f((1-i)s)$, $s\in{\mathbb{R}}$, which means that
\begin{gather*}
k=(1-i)\omega_1s\in \bar{Z}_u,\qquad -1\leq s\leq 1.
\end{gather*}
Since $e_2=0$ and $\bar{\lambda}=\lambda$ whenever $k=(1-i)\omega_1s$, $0<|s|\leq 1$, Proposition~\ref{Prop:asymp} ensures that $[0,\infty)$ coincides with the spectral arc extending to inf\/inity.

The expression \eqref{endpts} for the endpoints of the f\/inite spectral arc follows directly from~\eqref{vals} and Proposition~\ref{Thm:cplxLameThm}. Following the same line of reasoning as in the proof of Proposition~\ref{Prop:asymp}, and using standard power series expansions of $\zeta(z)$ and $\wp(z)$ around~$z=\omega_j$, we f\/ind parameterisations such that
\begin{gather*}
\lambda(t)=-\omega_2^2e_j+\big(3e_j^2-g_2/4\big)\big(\overline{\eta_2}\omega_2+\overline{e_j}|\omega_2|^2\big)^2t+\mathcal{O}\big(t^2\big),\qquad t\to+0.
\end{gather*}
(In fact, this requires no reality assumptions on $\mathcal L$ and any half-period $\omega$ and corresponding $\eta$ can be substituted for $\omega_2$ and $\eta_2$, respectively.) Specialising to the lemniscatic case, we readily deduce~\eqref{tngnts}.

The fact that the intersection point $\lambda^*=-\omega^2\wp(k^*)=\eta_2\omega_2$ is not an endpoint follows from explicit expressions for $\eta_2$, $\omega_2$ and $e_j$, see, e.g.,~\cite{DLMF}.
\end{proof}

We turn now to the proof of Theorem \ref{Thm:rect}. Since $\bar{Z}_u$ is invariant under the ref\/lection $k\to -k$, we may and shall restrict attention to that part of the fundamental period-paralellogram given by $\operatorname{Re} k\geq\operatorname{Im} k$. From \eqref{v} and the Legendre relation, we infer that
\begin{gather}\label{vValues}
v(\overline{\omega_2})=0,\qquad v(\omega_1)=v(-\omega_3)=\frac{\pi}{2}.
\end{gather}
Given that $k=\pm k^*$ are the only critical points of $f(k)$, the function $v(k)$ must be monotone on each component of~$Z_u$. Moreover, since $f^\prime(k^*)=0$ and $f^{\prime\prime}(k^*)\neq 0$, we can f\/ind neighbourhoods $U\ni k^*$, $V\ni 0$ and a biholomorphic mapping $g\colon U\to V$ such that
\begin{gather*}
f\big(g^{-1}(z)\big) = f(k^*)+z^2,\qquad z\in V,
\end{gather*}
which yields
\begin{gather*}
u\big(g^{-1}(z)\big)=x^2-y^2,\qquad v\big(g^{-1}(z)\big)=\operatorname{Im}[f(k^*)]+2xy,\qquad z=x+iy\in V.
\end{gather*}
Hence, in the neighbourhood $U$ of $k=k^*$, the zero set $\bar{Z}_u$ is given by $x=\pm y$. From these observations and \eqref{vValues}, we infer that $v(k)$ is strictly decreasing and increasing as $k\to k^*$ along the curve connecting $0$, $\bar{\omega_2}$ and $\omega_1$, $-\omega_3$, respectively, and that $0<v(k^*)<\pi/2$. In other words, all remaining solutions of~\eqref{v} are located on the component extending from $k=k^*$ towards $k=0$.

This result is readily generalised to all $\tau:=\omega_3/\omega_1\in i(0,\infty)$. Indeed, when we alter the value of $\tau$ from $\tau=i$ two things can happen: either the two curves remain intersecting or they break up into two disconnected curves. In the former case we can follow the same line of reasoning as in the lemniscatic case, and in the latter case we need only note that the closed spectral gaps cannot move from one curve to the other, since their locations depend continuously on~$\tau$. Applying the mapping $k\to\lambda=-\omega^2\wp(k)$, we thus conclude the proof of Theorem~\ref{Thm:rect}.

\begin{Remark}We believe that the intersection of the spectral arcs in the rectangular case with $\omega=\omega_2$ takes place only in the lemniscatic case, making it the only case in which the non-degeneracy conditions are violated, but we do not have a rigorous proof of this.
\end{Remark}

Consider now the case $\Delta<0$, corresponding to a rhombic period lattice~$\mathcal L$. We choose basic half-periods $\omega_1$, $\omega_3$ of the form
\begin{gather}\label{o13rhomb}
\omega_1\in (0,\infty),\qquad \operatorname{Re} \omega_3=\frac{1}{2}\omega_1,\qquad \operatorname{Im} \omega_3>0,
\end{gather}
so that $\overline{\omega_3}=\omega_1-\omega_3$.

Using the software {\it R}, we have in Fig.~\ref{Fig:rhombSpec} plotted the spectrum of the complex Lam\'e opera\-tor~\eqref{cplxLameOp} in four dif\/ferent rhombic cases, see Fig.~\ref{Fig:rhombkSpec} for the corresponding sets~$\bar{Z}_u$ and period lat\-ti\-ces~$\mathcal{L}$. These four cases exemplify the dif\/ferent types of spectra that occur for rhombic period lattices. In the lower left plot, we have the pseudo-lemniscatic case; and the upper left and lower right plots correspond to the two real forms of the complex equianharmonic curve.

\begin{figure}[t]\centering
\includegraphics[width=0.75\textwidth]{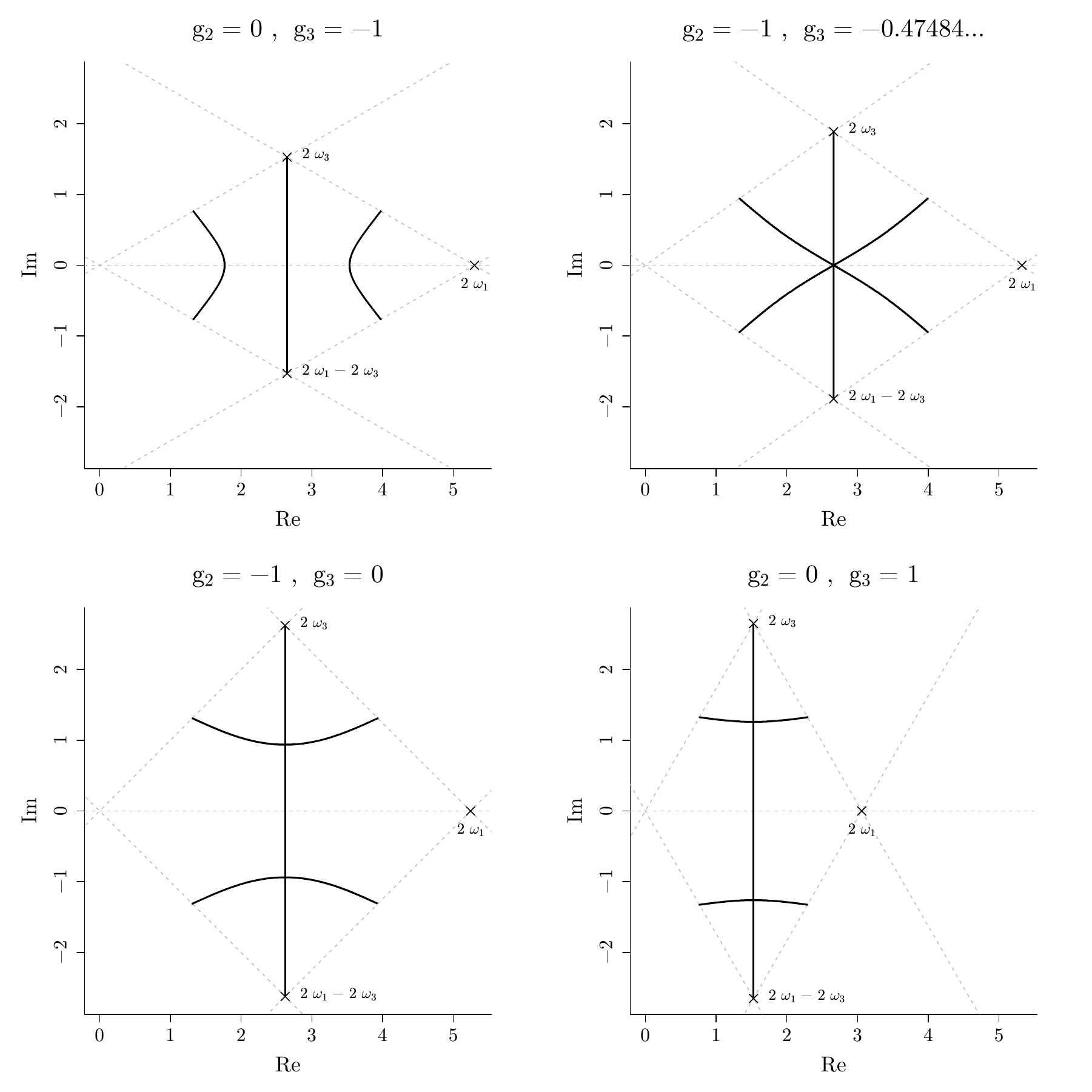}
\caption{The set $\bar{Z}_u$ for rhombic period lattices ${\mathcal L}$, $m=1$ and $\omega=\omega_1$.}\label{Fig:rhombkSpec}
\end{figure}

\begin{Proposition}
In the rhombic case \eqref{o13rhomb} the spectrum of the complex Lam\'e operator~\eqref{cplxLameOp} with $\omega=\omega_1$ is invariant under $\lambda\to\bar{\lambda}$, and the spectral arc extending to infinity coincides with $[-\omega_1^2e_1,\infty)$. Moreover, all closed spectral gaps are located on the spectral arc extending to infinity.
\end{Proposition}

\begin{proof}
Recalling that the spectrum is given by
\begin{gather*}
\lambda=-\omega_1^2\wp(k),\qquad k\in Z_u,
\end{gather*}
invariance under complex conjugation follows immediately from
\begin{gather*}
\zeta(\bar{k})=\overline{\zeta(k)},\qquad \wp(\bar{k})=\overline{\wp(k)}.
\end{gather*}
Combining these identities with $2\omega_3-\omega_1\in i(0,\infty)$, we see that
\begin{gather*}
k=(2\omega_3-\omega_1)s\in\bar{Z}_u,\qquad -1\leq s\leq 1.
\end{gather*}
Since $\bar{\lambda}=\lambda$ for $k=(2\omega_3-\omega_1)s$, $0<s\leq 1$, Proposition~\ref{Prop:asymp} and $2\omega_3-\omega_1=\omega_1$ $(\operatorname{mod}~\mathcal{L})$ imply that $[-\omega_1^2e_1,\infty)$ coincides with the spectral arc extending to inf\/inity.

Observing that the spectrum in the pseudo-lemniscatic case $\omega_1\in(0,\infty)$ and $2\omega_3=(1+i)\omega_1$ (see the lower left plot in Fig.~\ref{Fig:rhombSpec}) is identical to that in the lemniscatic case, the proof of Theorem~\ref{Thm:rect} is readily adapted to the rhombic case.
\end{proof}

Now let us study the corresponding non-degeneracy conditions
\begin{gather*}
\eta_1+\omega_1 e_j\neq 0, \qquad j=1,2,3.
\end{gather*}

\begin{Theorem}\label{Thm:nondeg}
In the rhombic case with $m=1$ the non-degeneracy conditions are violated for exactly one exceptional curve $\mathcal E_*$ uniquely determined by the condition
\begin{gather}\label{critCond}
\eta_1+\omega_1 e_1=0.
\end{gather}
The corresponding spectrum has a tripod structure with three simple analytic arcs joined at $2\pi/3$ angles. The same curve is the bifurcation point for the non-singularity condition \eqref{nonSing}, separating the cases with intersecting and non-intersecting spectral arcs.
\end{Theorem}

\begin{proof}
First of all, since both $\eta_1$ and $\omega_1$ are real the degeneracy condition $\eta_1+\omega_1 e_j=0$ implies that the corresponding $e_j$ must be real too, so $e_j=e_1$.

Recall now that $2\omega$ and $2\eta$ are the complete f\/irst and second kind elliptic integrals, which can be also written as integrals
\begin{gather*}
2\omega=\oint_\gamma \frac{{\rm d}E}{\sqrt{4(E-e_1)(E-e_2)(E-e_3)}},\qquad 2\eta=-\oint_\gamma \frac{E{\rm d}E}{\sqrt{4(E-e_1)(E-e_2)(E-e_3)}}
\end{gather*}
over the corresponding closed contour $\gamma$ in the complex domain. In our case the contour $\gamma$ is going around $e_1$ and $\infty$, or, equivalently, around the two complex conjugated roots $e_2$ and $e_3$. The degeneracy condition~(\ref{critCond}) can be rewritten then as{\samepage
\begin{gather}
\label{critCond2}
\oint_\gamma \frac{(E-e_1){\rm d}E}{\sqrt{4(E-e_1)(E-e_2)(E-e_3)}}=0,
\end{gather}
which also appeared in \cite{BG}.}

We have to show that there exists exactly one such curve. It would be convenient to use the parametrisation of the rhombic curves by f\/ixing $e_2=i$, $e_3=-i$ and $e_1=t\in \mathbb R$ (Weierstrass's restriction $e_1+e_2+e_3=0$ is violated, but this is not essential). We have to show that $\Phi(t)=0$ has the only real solution $t=t_*$, where
\begin{gather*}
\Phi(t):=\oint_\gamma \frac{(E-t){\rm d}E}{\sqrt{4(E-t)(E-i)(E+i)}}=\oint_\gamma \sqrt{\frac{E-t}{4(E^2+1)}}{\rm d}E.
\end{gather*}
Its derivative has the form
\begin{gather*}
\frac{{\rm d}}{{\rm d}t}\Phi(t)=-\frac{1}{2}\oint_\gamma\frac{{\rm d}E}{\sqrt{4(E^2+1)(E-t)}},
\end{gather*}
which is clearly negative since the integrand is positive on the half-line $E\in(t,\infty)$. Thus $\Phi(t)$ is strictly monotonically decreasing on the whole real line, which proves the uniqueness of the solution $t_*$. To show the existence it is enough to check that $\Phi(-2)>0$ and $\Phi(2)<0$ (e.g., numerically).

Numerical calculations give the value of the $j$-invariant
\begin{gather*}
j=1728\frac{g_2^3}{g_2^3-27g_3^2}
\end{gather*}
of the corresponding elliptic curve to be $j_*\approx 243.797$. One can see the corresponding period lattice $\mathcal L_*$ in the top right corner of Fig.~\ref{Fig:rhombkSpec}.

The tripod structure of the corresponding spectrum follows from a very general result by Batchenko and Gesztesy \cite[Theorem 1.1(v)]{BG}.
Its image is shown in the top right corner of Fig.~\ref{Fig:rhombSpec}.

The equal angles at the intersection point, which is also a general result due to Batchenko and Gesztesy \cite{BG}, is readily verif\/ied. Indeed, $\lambda$ is contained in the spectrum if and only if $\tau(\lambda)=\mathrm{tr} M(\lambda)\in [-2,2]\subset\mathbb{C}$, where $M(\lambda)$ is the monodromy matrix. Since $M(\lambda)$ depends analytically on $\lambda$, we have locally
\begin{gather*}
\tau(\lambda)=a_0+(\lambda-\lambda_0)^k+\cdots
\end{gather*}
for some $k\in\mathbb{N}$. If $a_0=\pm 2$, which corresponds to $\lambda_0$ being an end point of the spectral arc, we have an $k$-pod with angles $2\pi/k$, see e.g.~tripod in Fig.~\ref{Fig:rhombSpec}. If $a_0\in(-2,2)$, then we have $k$ arcs intersecting at $\lambda_0$ with angles between neighbouring arcs being $\pi/k$, see, e.g., the $k=2$ cases in the bottom two plots in Fig.~\ref{Fig:rhombSpec}.

To study the non-singularity condition recall that the critical points $k^*$ of the function $u(k)$ are given by $\eta + \omega\wp(k^*)=0$.
This means that the corresponding value of $\lambda(k^*)=-\omega^2\wp(k^*)=\eta\omega$ is real in this case. The non-singularity condition corresponds to the case when this value does not belong to the spectrum. Since the real part of the spectrum is $(-\omega^2e_1,\infty)$ it holds if\/f
$\eta\omega<-\omega^2 e_1$, which is equivalent to $\eta+\omega e_1<0$.
\end{proof}

\section[A qualitative analysis of the $m=2$ case]{A qualitative analysis of the $\boldsymbol{m=2}$ case}
Consider now brief\/ly what happens in the case $m=2$. The spectral curve of the corresponding Lam\'e operator
\begin{gather}\label{cplxLameOp2}
L=-\frac{{\rm d}^2}{{\rm d}x^2}+6\omega^2\wp(\omega x+z_0),
\end{gather}
has the form
\begin{gather}\label{specCurve}
w^2=\big(z^2-3g_2\big)\big(4z^3-9g_2z-27g_3\big).
\end{gather}

We restrict the attention to the rhombic case when $\Delta=g_2^3-27g_3^2<0$ and $\omega=\omega_1$. Setting $g_3=1$, this corresponds to $g_2\in(-\infty,3)$.

For $g_2\in(0,3)$, the roots of the f\/irst factor in the right-hand side of \eqref{specCurve}
\begin{gather*}
z^2-3g_2=0
\end{gather*}
yield two real endpoints $\lambda=\omega^2 z=\pm\omega^2\sqrt{3g_2}$, whereas the second factor
\begin{gather*}
4z^3-9g_2z-27g_3=0
\end{gather*}
gives one real positive endpoint as well as a complex conjugate pair of endpoints.

Hence we expect that for $g_2\in(0,3)$ the spectrum of~\eqref{cplxLameOp2} consists of two real intervals, one of which extends to inf\/inity, and one arc connecting the two complex endpoints. Moreover, as $g_2\to 0$, the f\/inite interval collapses to a double point. We can see this in the f\/irst two plots in Fig.~\ref{Fig:rhombSpec2}, produced numerically using the software~{\sl R}. Note that in the special case $g_2=0$, the spectrum consists of only two arcs.

For $g_2\in(-\infty,0)$, the f\/irst factor yields two purely imaginary endpoints $\lambda=\pm i\omega^2\sqrt{-3g_2}$, which suggests a spectrum consisting of two arcs each connecting two endpoints in the upper- and lower- half-plane, respectively, and a real interval extending to inf\/inity (see the last plot in Fig.~\ref{Fig:rhombSpec2}).

\begin{figure}[t]\centering
\includegraphics[width=0.75\textwidth]{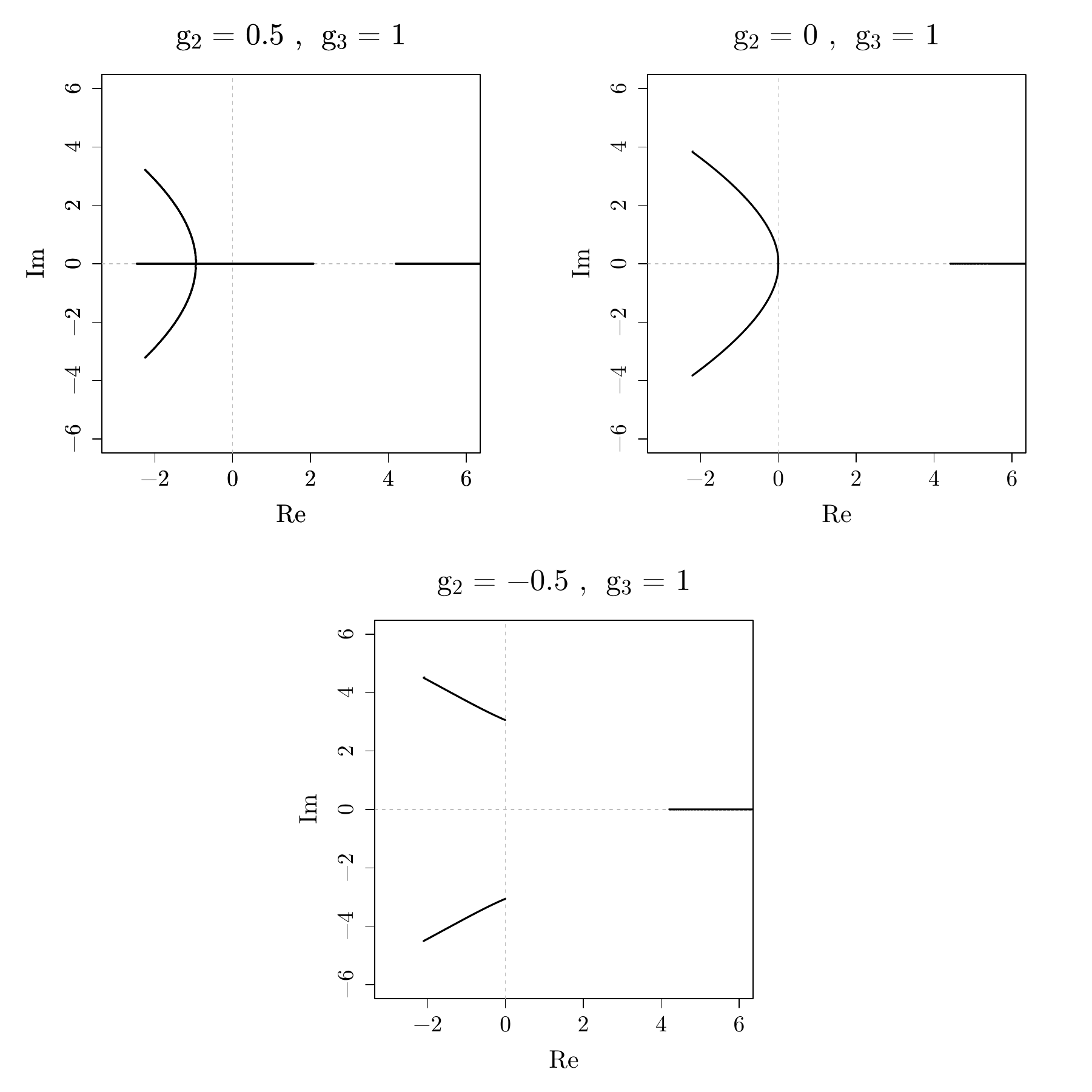}
\caption{Dif\/ferent types of spectra of the complex Lam\'e operator~\eqref{cplxLameOp2} with $g_3=1$, depending on the sign of $g_2 \in (-\infty,3)$.}
\label{Fig:rhombSpec2}
\end{figure}

\section{Concluding remarks}
The spectral theory of non-self-adjoint operators, in particular Schr\"odinger operators with complex potentials, has attracted signif\/icant attention in the last two decades, see, e.g.,~\cite{Davies2, Davies, Weikard}. Nevertheless, it is fair to say that it is still a far less complete theory than in the self-adjoint case. In particular, we feel there is a need for more examples, which can be studied explicitly (like the complex harmonic oscillator in \cite{Davies2}).

In our paper we used the example of the classical Lam\'e equation, whose general solution was found by Hermite in the 19th century. This example was already brief\/ly discussed in a much wider framework in \cite{BG, GW1}, but we believe that the complete picture was missing even in the simplest $m=1$ case. We used Hermite's explicit formulae and Morse theory ideas to study the structure of the spectrum of complex Lam\'e operators for two dif\/ferent types of real elliptic curves with rectangular and rhombic lattices. The rigorous analysis of higher $m$ cases seems to be dif\/f\/icult since the explicit formulae quickly become quite complicated as~$m$ grows, see \cite{BE, GV2, GV}.

A closely related interesting problem is to study the spectrum of the dif\/ference version of the Lam\'e operator
\begin{gather*}
 \mathcal L=\frac{\theta_1(x-m\eta)}{\theta_1(x)}T^\eta+\frac{\theta_1(x+m\eta)}{\theta_1(x)}T^{-\eta},
\end{gather*}
where $T^\eta$ is the shift operator def\/ined by $T^\eta\psi(x)=\psi(x+\eta)$, and $\theta_1(x,\tau)$ is the odd Jacobi theta function (see \cite{Zab} and references therein). When $\eta=\frac{p}{q}$ is rational we have an operator with periodic coef\/f\/icients, in general complex. We expect that the geometry of the corresponding spectrum (in particular, the band structure in the real case) depends on the arithmetic of~$p$ and~$q$.

\subsection*{Acknowledgements}
We are grateful to Jenya Ferapontov, John Gibbons and Anton Zabrodin for very useful and encouraging discussions, and especially to Boris Dubrovin, who many years ago asked one of us (APV) about the position of open gaps in the spectra of Lam\'e operators. We would like to thank Professor Gesztesy for his interest in our work and for pointing out further relevant references, including~\cite{BG} and~\cite{GW1}. The work of WAH was partially supported by the Department of Mathematical Sciences at Loughborough University as part of his PhD studies.

\pdfbookmark[1]{References}{ref}
\LastPageEnding

\end{document}